\documentclass{article}

\usepackage{lastpage}
\date{}

\usepackage{graphicx}
\usepackage{mathtools}
\usepackage{enumerate}
\usepackage{amssymb}
\usepackage{mathptmx}

\usepackage{latexsym}
\newtheorem{theorem}{Theorem}

\newtheorem{definition}{Definition}
\newtheorem{proof}{Proof}

\usepackage{amssymb}

\begin{document}

\title{Bolzano's measurable numbers: are they real?}

\author{Steve Russ and Kate\v{r}ina Trlifajov\'{a}\footnote{Steve Russ, Department of Computer Science, University of Warwick, Coventry CV4 7AL, UK, \emph{steve.russ@warwick.ac.uk}, 
Kate\v{r}ina Trlifajov\'{a}, Faculty of Information Technology, Czech Technical University in Prague, Th\'{a}kurova 9, 160 00 Prague 6, Czech Republic, \emph{katerina.trlifajova@fit.cvut.cz}.}}

\maketitle

\abstract{During the early 1830's Bernard Bolzano, working in Prague, wrote a manuscript giving a foundational account of numbers and their properties. In the final section of his work he described what he called `infinite number expressions' and `measurable numbers'. This work was evidently an attempt to provide an improved proof of the sufficiency of the criterion usually known as the `Cauchy criterion' for the convergence of an infinite sequence. Bolzano had in fact published this criterion four years earlier than Cauchy who, in his work of 1821, made no attempt at a proof. Any such proof required the construction or definition of real numbers and this, in essence, was what Bolzano achieved in his work on measurable numbers. It therefore pre-dates the well-known constructions of Dedekind, Cantor and many others by several decades. Bolzano's manuscript was partially published in 1962 and more fully published in 1976. We give an account of measurable numbers, the properties Bolzano proved about them, and the controversial reception they have prompted since their publication.}

\section{Introduction}

It is now widely accepted that any logically sound development of the limiting processes underlying mathematical analysis, or the calculus, requires the construction, definition or axiomatisation of the domain of real numbers. Or, at least, it requires some explicit assumption about the completeness of a linearly ordered field such that we can guarantee the closure of the field under limiting processes. Such recognition has been slow to be achieved. As late as 1908 the first edition of G.H. Hardy's classic textbook \emph{Pure Mathematics} simply assumed the rational and irrational numbers taken together had suitable algebraic and completeness properties. Only subsequent editions from 1914 gave a detailed account of Dedekind cuts as a construction of reals from the rationals. Both Dedekind and Weierstrass{---}leading professional mathematicians of their time{---}explicitly attributed motivation for their constructions of real numbers to the need, in the context of their teaching from the late 1850's, for a more rigorous basis for the differential and integral calculus.  

It is all the more remarkable therefore that a little-known Bohemian priest, who never did any formal teaching of mathematics, should have not only seen the need for such a foundation for calculus as early as 1817, but in the early 1830's had already gone a long way towards developing a bold, original, framework for a theory of real numbers. Bernard Bolzano's measurable numbers are developed in the final (7\textsuperscript{th}) section of his \emph{Reine Zahlenlehre} (\emph{Pure Theory of Numbers}), here abbreviated to RZ. It was written in three manuscript versions in the early 1830's that are briefly described by the late Jan Berg in (Bolzano 1976) where his transcription of the final version is published. It is the final (7\textsuperscript{th}) section of RZ that is our primary source and wherever paragraph numbers (of the form \S 1) appear in isolation in this chapter they refer implicitly to this final section of RZ. The purpose of this chapter is to outline the achievements in this part of Bolzano's work and to identify, as far as the evidence allows, how Bolzano regarded his measurable numbers. We also endeavour to throw light on the somewhat confused reception his theory has received since its publication. 
  
Bolzano's written output of published works  and unpublished manuscripts was prodigious. It is being given a comprehensive publication in the \emph{Bernard Bolzano Gesamtausgabe }(\emph{Complete Works}) with over 90 volumes of the projected 129 volumes having already appeared. We use the abbreviation BGA for this edition. The best overview of Bolzano's life and work in English is (Morscher 2008) which is particularly good on the work on logic and philosophy but includes useful material on mathematics too.  Some of the most important mathematical works have appeared in English translation, with brief commentary, in (Russ 2004). Much more sustained study of his mathematics appears in (Rusnock 2000) and, in French,  (Sebestik 1992). Bolzano's most important philosophical work, in four volumes, is \emph{Wissenschaftslehre} (\emph{Theory of Science}) now with a complete English translation (Bolzano 1837/2014).
  
In a deservedly much-admired paper of 1817 Bolzano presented an ingenious and original proof, by repeated bisection, of the intermediate value theorem, namely that a continuous real-valued function that changes its sign at the endpoints of the interval $[a,b]$ has a zero somewhere in the open interval $(a, b)$. In \S 7 of (Bolzano 1817) he articulates before Cauchy a general criterion for the convergence of an infinite sequence. And he does so more clearly and concisely than Cauchy and without reference to the infinitely large or the infinitely small as occurs in (Cauchy 1821). It is therefore historically accurate, as some authors are now doing, to refer to the \emph{Bolzano-Cauchy convergence criterion}. We state this here as follows:
\begin{quotation}
If a sequence of terms $a_1, a_2, a_3, \dots  a_n, \dots a_{n+r} ,\dots $ has the property that the difference between its $n$-th term an and every later one $a_{n+r}$, however far this latter term may be from the former, remains smaller than any given quantity if $n$ has been taken large enough, then there is always a certain \emph{constant quantity}, and indeed only \emph{one}, which the terms of this sequence approach and to which they can come as near as we please if the sequence is continued far enough.  (Bolzano 1817, \S 7) \footnote{This is a slightly modified form of  the theorem, the exact version is in (Russ 2004, 266).}
\end{quotation} 
Bolzano attempted a detailed proof of the sufficiency of this criterion. Cauchy simply assumed the criterion without comment. It is an irony that Bolzano is today criticised for his \lq flawed\rq \ proof while Cauchy is not criticised for making the silent assumption. As pointed out in (Stedall 2008, 496), Bolzano's was the \emph{only} attempted proof of the criterion in the early 19\textsuperscript{th} century. We can put \lq flawed\rq \ in inverted commas here because while the result is not true in the field of rationals alone, it is true in the field Hardy was assuming and evidence that this was also the field Bolzano intended comes from his remark near the end of the \emph{Preface} when he summarises the content of the body of his paper saying that, 
\begin{quotation} 
for anyone having a correct concept of quantity the idea of [the limit] is the idea of a real, i.e. an \emph{actual} quantity.
\end{quotation} 
The main work of his attempted proof is to show that the limit is 
\begin{quotation}
not impossible [because] on this assumption it is possible to determine this quantity as accurately as we please.\end{quotation}  
The possibility of some entity was always regarded by Bolzano as a necessary, but not sufficient condition for existence. In this case the possibility, combined with the need, led Bolzano to claim that with a \emph{correct} \ concept of quantity (or number, the two concepts were not sharply distinguished by Bolzano) the convergence criterion was indeed sufficient to ensure the existence of a limit. For detailed discussion of this proof see (Kitcher 1975, 247{--}251), (Rusnock 2000, 69{--}84) and (Russ 2004, 149). 
  
No-one could be in doubt that to maintain a correspondence between numbers and lengths (such as the diagonal of a unit square), or to locate the zeros of some functions (such as $x^2 - 2$), required the use of irrational numbers. What very few people seem to have appreciated in the early 19\textsuperscript{th} century was the logical need to give some construction or definition of the irrationals on the basis of the rationals. That was the problem giving rise to the celebrated solutions such as those of Dedekind, Weierstrass, M\'{e}ray, and Cantor published from 1872 (although originating in the late 1850's). It was this same problem that Bolzano had already addressed, largely successfully but unpublished and quite unknown, in his theory of measurable numbers of the early 1830's. It is the purpose of this chapter to give to give an outline of Bolzano's achievement in this work. The clearest source for the later work, which became a standard reference, is (Dedekind 1963). The original works of the other authors just mentioned can be hard to locate but are well-described, and referenced, in many works, for example in chapter IV of (Ferreiros 1999). 
  
During the 1840's Bolzano realised he would not complete the manuscripts for RZ and he passed these and many other mathematical papers on to Robert Zimmermann, one of his former students and the son of a good friend. Zimmermann gained a Chair in Philosophy in Vienna and deposited a locked suitcase with many of Bolzano's mathematical papers in the Austrian National Library. They languished there until discovered by Professor M. Ja\v{s}ek from Pilsen in the 1920's. Some material was published from the 1930's onwards but it was not until 1962 that a partial transcript of RZ was published as (Rychl\'{i}k 1962). This was in fact only a partial transcription of the final (7\textsuperscript{th}) section of RZ.  A much fuller publication of all sections of RZ including some parts of the 7\textsuperscript{th} section which were illegible to Rychl\'{i}k appeared in (Bolzano 1976).  

\section{Measurable Numbers}
\label{measurable}

The title of this final section of RZ is \lq infinite quantity concepts\rq \ but from his hand-written revisions, and alternative phrasing, it is clear he had difficulty deciding between \lq quantity\rq  and \lq number\rq \ on the one hand, and between \lq concept \rq \ and \lq expression\rq \ on the other hand. We shall not adhere exclusively to any of these phrases but try to keep in mind the ambiguity reflected in Bolzano's usages.  A regular explicit usage was that an infinite quantity \emph{concept} is represented by an infinite quantity \emph{expression}. The following are some of his examples of the latter:
$$1 + 2 + 3 + 4 + \dots \mbox{in inf.}$$
$$\frac{1}{2} - \frac{1}{4} + \frac{1}{8} - \frac{1}{16} + \dots \mbox{in inf.}$$
$$(1 -  \frac{1}{2})(1 - \frac{1}{4})(1 - \frac{1}{8})(1 - \frac{1}{16}) \dots \mbox{in inf.}$$
$$a  +   \frac{b}{1+1+1+\dots \mbox{in inf.}}$$ where $a, b$ is a pair of natural numbers \footnote{It is clear from Bolzano's usage in RZ that this is the appropriate translation of `wirkliche Zahlen'.}. 
The crucial property of an infinite quantity expression is that it should contain infinitely many operations of addition, subtraction, multiplication or division. But Bolzano is at pains to point out that this does not mean we have to have the ideas of all the components of such an expression, we could not. He says it is like the way we can describe, and designate, a pocket watch very simply and without having ideas of the many components inside the watch. So he says that \lq infinite\rq is here being used in a figurative way. The number concept itself is a single thing arising from the multitude of operations. He says that for every number expression $S$ (not only infinite ones), \begin{quotation}
... we \emph{determine by approximation, or measure} [the number expression $S$], if for every positive integer $q$ we determine the integer $p$ that must be chosen so that the two equations
$$ S = \frac{p}{q} + P_1 \textnormal{\ \  and\ \ }S = \frac{p+1}{q} - P_2 $$
arise, in which $P_1$ and $P_2$ denote a pair of strictly positive number expressions (the former possibly being zero). (RZ, \S 6), (Russ 2004, 361). \end{quotation}   

A \emph{strictly positive number expression} is one in which{---}according to the earlier section 4 of RZ{---}\emph{no subtraction appears}. When these two equations are always satisfiable Bolzano says, we determine the number expression $S$, as precisely as we please and then he calls such an expression a \emph{measurable} expression (RZ, \S 6).  The fraction $p \over q$ is then called the \emph{measuring fraction} for $S$.

It appears that we can visualise a measurable expression or number by imagining a line of length $S$ which is being measured by rulers with units divided into $q$ equal parts, a $q$-ruler. For a given ruler either the line will match up exactly with some division marks (and so $S$ is rational), or the line is always strictly in between two divisions. In either case $S$ is measurable. Such a visualisation assumes a close correspondence between a geometric continuum and an arithmetic continuum. In fact the theme of continuum{---}whether geometric, arithmetic, temporal or physical{---}is a powerful underlying theme in much of Bolzano's work from his first publication with his analysis of straight line in 1804 to the explicit treatment of continuum in \S 38 of his well-known \emph{Paradoxes of the Infinite} published as (Bolzano 1851) shortly after his death.

This definition of an expression $S$ as measurable implies that for every positive integer $q$ there is a $p$ so that 
				$$\frac{p}{q} \leq S  <  \frac{p+1}{q}$$
and it means that Bolzano is explicitly associating a measurable number with infinitely many approximating rational intervals. It is easy to show by experiment that these intervals are not nested (readers may like to try this with examples of $S = 2/3$, or $S = \sqrt 2$ ) but it is not difficult{---}another exercise for the reader!{---}to prove that they are \emph{dually directed} in the sense that the intersection of any two (say for $q$ and $q'$ with $q' > q$) is a proper subset of some interval for $q'' $  with $q''  > q' $. But the association of a measurable number with an infinite collection of intervals is significant. The use of infinite collections to define real numbers is a feature common to all  the major constructions that emerged in the later 19\textsuperscript{th} century. 
  
In RZ \S 21 Bolzano introduces the expression
					$$S = \frac{1}{1 + 1 + 1 + \dots \mbox{in inf.}}$$
					
After showing $S$ is measurable with $p= 0$, however large $q$ may be taken,  Bolzano goes on to conclude that  
\begin{quotation} ... we are not justified, at least by the concepts so far, in considering the expression $S$ to be equivalent to $0$. $\dots$ This is an example of an \emph{infinitely small} positive number. \end{quotation} 

This represents a major change in Bolzano's view on the infinitely small from some 15 years earlier when he declared that 
\begin{quotation} ... calculus is based on the shakiest foundations $\dots$ on the self-contradictory concepts of infinitely small quantities. (Bolzano 1816, Preface)\end{quotation}
It would be an interesting project to investigate what influenced his change of mind here. There are many potential sources for this study such as the series \emph{Miscellanea Mathematica} of the BGA (Bolzano's mathematical diaries that he maintained for most of his life), the sustained reflection on infinite collections in (Bolzano 1851), and the correspondence in the series III of the BGA, as well as numerous further archival sources. 
  
The possibility of infinitely small numbers that are measurable gives rise to a crucial revision of his definition of the equality of measurable numbers. In the long discussion of section \S 54, and again as an explicit definition in \S 55, Bolzano states that measurable expressions are equal to one another if for every arbitrary $q$ always one and the same $p$ may be found yielding a measuring fraction $p\over q$ common to both. Otherwise stated,  two numbers were equal if  \begin{quotation} 
they behave in the same way in the process of measuring. \end{quotation}
However, Bolzano inserted a highly significant revision near the beginning of \S 54. He observed that the section needs to be re-written because 
\begin{quotation} numbers which differ only by an infinitely small [amount] can behave differently in the process of measuring. \end{quotation}  
He gives the example of $1$ and $1 - \frac{1}{1 + 1+ 1+ \dots \mbox{in inf.}}$ where the former, for every $q$, has measuring fraction with $p = q$, but the latter has $p = q - 1$. So he revises the definition to say that if the pair of numbers $A$ and $B$ have a difference $A - B$ which in the process of measuring always has a measuring fraction of the form $0/q$ then $A = B$.  And if the difference is positive then $A > B$, if it is negative then $A < B$. Unfortunately the section \S 54, and subsequent sections have not, in fact, been revised other than by way of this inserted note. In \S 56 Bolzano says it is not so much the notion of equality which is extended by his definition but rather it is the object (here measurable number) which is affected. The earlier definition discriminates more finely than the new one which, in modern terms, is an equivalence relation.  We shall speak here of \emph{old} measurable numbers and \emph{new} measurable numbers. For the old numbers equality was not transitive, for the new numbers equality is transitive (and reflexive and symmetric). Perhaps Bolzano was one of the first to see the need for an equivalence relation, and the related equivalence classes, or factor-classes. It is interesting to note that in (Klein 1908) two separate themes in the development of analysis are identified: the Weierstrassian approach (in the context of an Archimedean continuum), and an approach with indivisibles and/or infinitesimals (in the context of a richer non-Archimedean continuum).  This is cited at the beginning of the paper (Bair et al 2008) which continues with extensive commentary but, curiously, no mention of Bolzano whose two definitions of the equality of measurable numbers clearly straddle both approaches described by Klein.

Bolzano proves important properties of the ordering of measurable numbers in \S\S 61 - 79.  We summarise these in the following using modern notation and giving the relevant paragraphs of RZ in parentheses after each result.

\begin{theorem}

Let $A, B, C$ be measurable numbers. 

\begin{enumerate}
\item Transitivity. $((A > B) \wedge (B > C)) \Rightarrow (A > C). \;\;$(\S 63 ) 
\item Linearity. $(A = B \vee A > B \vee A < B). \;\; $ (\S 61, 73) 
\item Unboundedness. $(\forall A)(\exists B)(\exists C)((B < A) \wedge  (A < C)).  \; \;$(\S 70) 
\item Density. $(A < C) \Rightarrow (\exists B)((A < B) \wedge (B < C)). \; \;$ (\S 79) 
\item Archimedean property. $(\exists n)(\frac{A}{n} < B < n \cdot A). \;\; $ (\S 74) 
\item $(A > B) \Rightarrow (A + C) > (B + C). \; \;$(\S 67)
\end{enumerate}
\end{theorem}

The next results are about the arithmetic properties of measurable numbers, they are in \S\S 45, 51, 59, 99-121 and are gathered here as follows. 

\begin{theorem}

Let $A, B, C$  are measurable numbers. 

\begin{enumerate}

\item Closure under addition. $A+B$ is a measurable number. $\;\;$ (\S 45) 
\item Closure under multiplication. $A \cdot B$ is a measurable number. $\;\;$ (\S 59, \S 45, \S 51) 
\item $B \neq 0 \Rightarrow \frac{A}{B}$ is a measurable number. $\;\;$ (\S 111)
\item Property of $0$. $A \cdot 0 = 0 \cdot A = 0. \;\;$ (\S 71) 
\item Associativity of multiplication. $A \cdot (B \cdot C) = (A \cdot B) \cdot C. \;\; $(\S 99) 
\item Commutativity of multiplication. $A \cdot B = B \cdot A. \;\;$ (\S 99) 
\item Distributivity. $A \cdot (B + C) = A \cdot B + A \cdot C. \;\; $(\S 101)
\item Rules for fractions such as $(A = B \wedge C \neq 0) \Rightarrow \frac{A}{C} = \frac{B}{C}, B \neq 0 \Rightarrow \frac{A}{B} \cdot B = A, etc. \;\; $(\S\S 113 - 121)

\end{enumerate}

\end{theorem} 
Bolzano can now finally prove the sufficiency of the \emph{Bolzano-Cauchy} convergence criterion in \S 107. It is what (Rusnock 2000) deservedly calls the \emph{Bolzano-Cauchy theorem}. He stated it already in (Bolzano 1817/2004), but could not prove the existence of the relevant limit. The following much improved formulation and proof (Russ 2004, 412) is in terms of a sequence of measurable numbers: in modern terms it means that the ordered field of measurable numbers is complete.    

\begin{theorem} 

Suppose the infinitely many measurable numbers $X^1, X^2, X^3, \cdots $ proceed according to such a rule that the difference $X^{n+r} - X^n$, considered in its absolute value always remains smaller than a certain fraction $\frac{1}{N}$ which itself can become as small as we please, providing the number $n$ has first been taken large enough. Then I claim there is always one and only one single measurable number $A$, of which it can be said that the terms of our series approach it indefinitely. 
\end{theorem}
Bolzano distinguishes three cases:  the sequence  $X^1, X^2, X^3, \cdots $ is non-decreasing, non-increasing or alternating. He begins with non-decreasing sequences. He proves that if there is a limit of the sequence then for every $q$ there is a $p$ such that $\frac{p}{q}$ is a measuring fraction of the (conjectural) limit. So we have a complete set of measuring fractions. Bolzano shows that this limit is determined uniquely. The proof for non-increasing sequences is similar. From alternating sequences we can choose subsequences which are either non-decreasing or non-increasing and their limit is the limit of the whole sequence. The proof is long and not easy to follow in some of the details but we support Rusnock's opinion in concluding positively on the logical structure of this proof in (Rusnock 2000, 188).

This was a vindication of Bolzano's belief, announced in the \emph{Preface} of his (Bolzano 1817/2004),  that with a `\emph{correct} concept of number' the convergence criterion quoted there was indeed sufficient to establish a limit number. That correct concept can, we propose, be identified with the measurable numbers. This concludes the demonstration that, in spite of what we would now regard as some gaps and confusions,  Bolzano's domain of measurable numbers is a complete linearly ordered field and so isomorphic to the real numbers as we know them today.

\section{Are Measurable Numbers really Real?}
\label{really}

It is an irony that the late Bob van Rootselaar, who did so much work in careful transcription and editing of Bolzano's mathematical diaries, was also one of the severest critics of Bolzano's work on real numbers. Very soon after the publication of (Rychl\'{i}k 1962) there appeared (van Rootselaar 1963) in which the author declares in the opening two pages that, 
\begin{quotation}
Bolzano's elaboration [of measurable numbers] is quite incorrect, and that the more advanced part of Bolzano's theory is inconsistent.
\end{quotation} 
One might have supposed this to be largely due to the fact that it only came to light in (Bolzano 1976) how some significant improvements in the content of Bolzano's work were revealed by the more thorough and detailed reading of the manuscript version by Berg. For example, the revision to the equality criterion mentioned above was not legible for Rychl\'{i}k and was omitted, and other parts that were deleted by Bolzano, were included by Rychl\'{i}k. But this is evidently not the cause of van Rootselaar's negative views. He refers, long after publication of Berg's transcription, to the obscure, but interesting, work (Ide 1803) and concludes that 
\begin{quotation}Both the theories [of Ide and Bolzano] presuppose the existence of the real numbers$ \dots$ progress is made only with a theory such as that of Cantor. (van Rootselaar 2003).
\end{quotation}
On the contrary, we wish to support here the claim that Bolzano did not{---}in contrast to the situation in 1817{---}presuppose the existence of the reals at all and that, in fact, his theory has a close resemblance to that of Cantor and to that of Dedekind.
  
Soon after the criticism of van Rootselaar came a strong rejoinder in (Laugwitz 1965) pointing out that it only needed a small change in the definition of infinitely small quantity in order to rectify many of Bolzano's proofs and results.  After such a change it can indeed be viewed as a consistent theory of real numbers. Following the publication by Berg of the improved reading of Bolzano's work it was discovered that Bolzano had already made the change that Laugwitz recommended and so (Laugwitz 1982) was able to fully endorse Bolzano's work as a theory of the real numbers.
  
Subsequently there have been a variety of commentators with a wide spectrum of views on Bolzano's work on measurable numbers, certainly enough to justify the remark that, 
\begin{quotation} there is perhaps no area of Bolzano's research about which there is less agreement than his theory of real numbers$\dots$ Bolzano's analyses were a preamble to his theory of measurable numbers, which is itself a tangled thicket of issues, much disputed in the literature  (Simons 2003, 118).
\end{quotation} 
We cannot give here any comprehensive study of the debates but should refer at least, further to those already mentioned, the detailed studies in (Spalt 1991), (Sebestik 1992) and (Rusnock 2000). The last-cited work sums up the difficulties well: \begin{quotation} Bolzano's theory of measurable numbers as it has come down to us is obviously in fairly rough shape….\end{quotation}   
Then after some valuable detailed discussion of the proof in RZ \S 107 he concludes,  \begin{quotation} . . . on the essential point of conceptual structure, Bolzano was almost entirely successful in characterizing the reals. (Rusnock 2000, 184-188). 
\end{quotation}
So much for the commentary in the literature at a high level. We now give some indication of the detailed discussion and interpretation of what Bolzano was doing in this rich final section of RZ and the direction of our own thinking on the main themes. 
  
The motivation for the title of the final section of RZ as \emph{Infinite Quantity Concepts}  (or also \emph{Infinite Quantity Expressions}) is clear. The previous three sections of the work were all concerned explicitly with rational numbers. Number expressions with only finitely many arithmetic operations will only yield rational number results. So to address the construction of irrational numbers it was essential to consider expressions with infinitely many operations. Some commentators seem to have been distracted into a focus on how best to interpret such expressions when the central, over-riding concept{---}that of \emph{measurable} number{---}is one which applies to both infinite and finite number expressions. And a measurable number is defined in terms of two equations or, equivalently, by an approximating interval as described above. 
  
The problem, and the need for some interpretation, with the notion of infinite quantity expression arises because it appears to be very general (just requiring infinitely many arithmetic operations) but the examples given by Bolzano are rather simple.  It appears to allow for expressions such as continued fractions, compounds of multiple continued fractions or any arbitrarily complex expression, even whether or not there is any evident rule for continuing the expression. But the most complex one used by Bolzano in RZ is the one occurring near the end of \S 5:
$$\frac{s (1 + 1 + 1 + \cdots \mbox{in inf.}) - qb}{q (1 + 1 + 1 + \cdots \mbox{in inf.})}.$$	
For the proof in this paragraph it is only claimed that this is a positive quantity: it does not need to be evaluated. Ladislav Rieger, one of the editors of  (Bolzano 1962), suggests  in his \emph{Vorwort} that such infinite number expressions might be interpreted as \lq symbols for effectively described, unbounded, computational procedures on rational numbers.\rq \   But for what procedure is it a symbol? In this case it might be 'natural' to say we obtain a partial $n$-th value if we sum the first $n$ terms of numerator and the first $n$ terms of the denominator, then divide. But it is easy to construct cases where there is no such natural rule. van Rootselaar attributes the idea of interpreting infinite number expressions as infinite sequences of rational numbers to Rychl\'{i}k and takes up the idea himself with enthusiasm even declaring, 
\begin{quotation}. . . indeed anyone who reads Bolzano's manuscript is bound to accept it [Rychl\'{i}k's interpretation] (van Rootselaar 1963, 169). \end{quotation}
So the expression $\frac{b}{1+1+1 + \dots \mbox{in inf.}}$ is interpreted as the sequence  $\{\frac{b}{n}\}$, and the expression $1 + 2 + 3 + \dots \mbox{in inf.}$  corresponds to the sequence $\{\frac{1}{2}n(n+1)\}$. It might be noted in the former case that $b \over n$ is not strictly a partial sum of the expression (it is the reciprocal of a partial sum) but it is a 'partial computation' to use Rieger's phrase.

Van Rootselaar develops his own elaborate sequence interpretation and uses it to give an interpretion of Bolzano's measurable numbers. Having expressed the measurable number $S$ in terms of an infinite sequence $\{s_n\}$ each term of the sequence is then assigned an approximating interval using terms involving triple subscripts and equations of the form
$$s_n  = \frac{p_q(S)}{q}  + P_{q,1,n}  = \frac{p_q(S) + 1}{q } - P_{q,2,n}$$
where full details are given in (van Rootselaar 1963, 173). Apart from being rather cumbersome a strong argument against carrying the sequence interpretation to such lengths is the one put forward in (Becker 1988). \footnote{This is an unpublished dissertation which we have not seen but rely on the report of it in (Spalt 1991).} Here Becker simply points out the obvious fact that Bolzano, although being fluent at working with infinite sequences, nowhere suggests that he was himself making use of a sequence interpretation for either infinite number expressions or measurable numbers.  However, there is an important aspect of what van Rootselaar is doing in the above formulation which we shall ourselves shortly be endorsing. That is, he is making a very tight association between a measurable number and an infinite collection of approximating intervals. Whether or not this was in Bolzano's mind was never explicitly stated by him. We can only judge from the surrounding context and his usage.
  
Returning to the more limited application of the sequence interpretation{---}that for simple infinite series{---}there is a well-known problem arising in the case of an alternating series in which the partial sums are non-monotonic. The problem occurs already with the expression 
$$S = \frac{1}{2} - \frac{1}{4} + \frac{1}{8} - \frac{1}{16} + \dots \mbox{in inf.}$$ 
which Bolzano himself presents as an example of an infinite number expression in \S 2. If we interpret it as the sequence of partial sums we obtain the sequence $\frac{1}{2}, \frac{1}{4}, \frac{3}{8}, \frac{5}{16}, \dots $ that converges to $\frac{1}{3}$. This is a non-monotonic sequence, the terms are sometimes less, and sometimes greater, than $\frac{1}{3}$. It is impossible to say that for $q = 3$  there is $p$ such that the sequence lies within the interval $[\frac{p}{3}, \frac{p+1}{3})$. It is similar for $q = 3n$ where $n$ is a positive integer. The question is: did Bolzano consider this expression, and consequently all convergent alternating series, as measurable numbers? There are several possibilities:
\begin{enumerate} 
\item Bolzano did not regard the expression $S$ as a measurable number because it does not satisfy his definition of measurability (Sebestik 1992, 370).  Bolzano  generally did not regard sequences oscillating around a rational number as measurable. 
Then there is a problem. Bolzano had proved in \S 45 that measurable numbers are closed under addition. But for instance if we take the two measurable numbers 
$$A\ =\ 2 - 2 + \frac{1}{2} - \frac{1}{2} + \frac{1}{8} - \frac{1}{8} + \dots \mbox{in inf.} $$
$$B \ =\ -1 + \frac{1}{2} + \frac{1}{4} + \frac{1}{8} +  \dots \mbox{in inf.}$$  
then their sum  
$$A+B\ =\ 1 - \frac{3}{2} + \frac{3}{4} - \frac{3}{8} + \frac{3}{16} - \dots \mbox{in inf.}$$  
is not measurable. 
\item Bolzano regarded $S$ as measurable. He speaks about measurable numbers as quantities which we can measure up to $\frac{1}{q}$ for every $q$. And the expression $S$ does have this property. If we return to the picture of a $q$-ruler we see that if we are allowed to shift the ruler then $S$ can be enclosed between two scale divisions. Otherwise, if we do not shift the ruler, the expression value oscillates around one division mark, successively occupying two adjacent intervals. 

If we admit in the definition of measurable numbers that $p$ can be a rational number then the definition would be: $S$ is measurable if for all $q$ there is a \emph{rational} number $r$ and two positive expressions $P_1, P_2$ such that
$$\frac{r}{q}  +  P_1 =  S = \frac{r + 1}{q}  - P_2.$$ 
Or we could repair the definition in this way: $S$ is measurable if for all $q$ there is a \emph{rational} number $r$ and two positive expressions $P_1, P_2$ such that
$$ r  +  P_1 =  S = r  + \frac{1}{q}  - P_2.$$ 
In the both cases $A+B$ would be measurable and generally the sum of two measurable numbers would be measurable too.

\item  Bolzano regarded $S$ as measurable but he had a different concept of an infinite calculation. He considered $S$ as one exactly given quantity which is equal to $\frac{1}{3}.$  He generally considered number expressions with oscillating values and approaching a rational number as being equal to that rational number. There is some indication of this in his proof that the sum of two measurable numbers $A, B$ is measurable in $\S 45.$  Bolzano analyses several cases. The last case is about number expressions which could be interpreted as non-monotonic sequences. Bolzano obtains after many equations the expression  
$$A + B = \frac{p^1 + p^2 + 1}{q} + P^{13} - \Omega^1 =  \frac{p^1 + p^2 + 1}{q} - P^{14} + \Omega^2$$
where by $\Omega^1$  and  $\Omega^2$ Bolzano understands a pair of number which can decrease indefinitely. Therefore 
$$\Omega^1 + \Omega^2 = P^{13} + P^{14}$$ Because $\Omega^1 + \Omega^2$ can decrease indefinitely one can say that also $P^{13}$ and $P^{14}$ can decrease indefinitely. The sum $A + B$ evidently alternates and approaches a fraction $\frac{p^1 + p^2 + 1}{q}$. Bolzano, referring to a similar result for rational numbers (\S 8, 6th Section), says  that
$$A + B = \frac{p^1 + p^2 + 1}{q}$$   
\end{enumerate}

It is hard to know what was in Bolzano's mind here. The manuscript that remains was not a definitive version. In order to deal with the case of oscillating values Laugwitz demonstrated that it suffices to change the equation in Bolzano's definition of measurable numbers to repair the theory (Laugwitz 1982, 407). We give his proposal as follows in a slightly modified form and call it the \emph{Laugwitz condition}.

\begin{definition}\label{Laugwitz}
An infinite number concept $S$ is a \emph{measurable number} if for every positive natural number $q$ there is a natural number $p$ and two positive number expressions $P_1$ and $P_2$ such that the two following equations are satisfied:
$${p-1 \over q} + P_1 = S = {p+1 \over q} - P_2. $$ 
\end{definition}
It is in the spirit of Bolzano's idea that the infinite number expression $S$ is measurable if 
\begin{quotation}
the determination by approximation, or the measurement of $S$, can be carried out as precisely as we please. (RZ \S 6)
\end{quotation}
Measurable numbers under this condition are closed under addition. All other theorems remain true. \footnote{Another interesting possibility for repairing the definition of measurable numbers is in (Rusnock 2000, 185 - 186).} Bolzano himself suggested a generalisation of this modification in \S 122 (Russ, 2004, 428) where he writes: 
\begin{quotation}Perhaps the theory of measurable numbers could be simplified if we formulated the definition of them so that $A$ is called measurable if we have two equations of the form $A\ = \frac{p}{q} + P \ = \ \frac{p+n}{q} - P$, where for the identical $n,  q$ can be increased indefinitely.
\end{quotation}

We return now to the discussion of the sequence interpretation of infinite number concepts which we began earlier in this section (on p.10). It is surely a reasonable proposal that if we associate a number concept $P$  with the sequence $\{p_n\}$ then $P \geq 0$ if and only if there exists $N$ such that $p_n \geq 0$ for all $n \geq N$ (Rusnock 2000, 185).  Then adopting the Laugwitz condition as above, and using modern notation, we shall prove the following.
\begin{theorem} If a rational sequence $\{a_n\}$ represents the infinite number expression $S$ then $\{a_n\}$ satisfies the Bolzano-Cauchy convergence criterion (we call it a BC-sequence) if and only if $S$ is a measurable number. 
\end{theorem}
\begin{proof}

Let $q, m, n, k, p$ be natural numbers, $P_1, P_2$ are strictly positive number concepts.

By the Laugwitz condition, the infinite number concept $S$ is measurable if
$$(\forall q)(\exists p)(\exists P_1)(\exists P_2)({p-1 \over q} + P_1 = S = {p+1 \over q} - P_2).$$ 
It means in our interpretation that 
$$(\forall q)(\exists p)(\exists m)(\forall n > m)({p-1 \over q} \ < \ a_n \ < \ {p+1 \over q}).$$
Remember that $\{a_n\}$ with each $a_n \in \mathbb{Q}$ is a BC-sequence iff 
$$(\forall k)(\exists m)(\forall n > m) |a_n - a_m| < {1 \over k}.$$

(i) Let  $S$ be measurable. We will prove that $\{a_n\}$ is a BC-sequence. Take any $k$. Let $q = 2k$. Then $$(\exists p)(\exists m)(\forall n > m)({p-1 \over 2k} \ < \ a_n \ < \ {p+1 \over 2k}).$$ 

Hence $$(\forall n > m)|a_n - a_m| < |{p+1 \over 2k} - {p-1 \over 2k}| \ = \ {1 \over k}.$$  
(ii) Conversely, let $\{a_n\}$ be a BC-sequence. We will prove that $S$ is a measurable number. Take any $q$. Let $k = 2q$. Then $$(\exists m)(\forall n > m)(|a_n - a_m| < {1 \over 2q}.$$
We know that $a_m \in \mathbb{Q}$ therefore $(\exists r)({r \over 2q} \leq a_m < {r+1 \over 2q}$. Hence ${r-1 \over 2q} = {r \over 2q} - {1 \over 2q}\ <\ a_n\ <\ {r+1 \over 2q} + {1 \over 2q} = {r+2 \over 2q}$. If $r$ is even take $p$ such that $r = 2p$ and if $r$ is odd take $p$ such that $r = 2p - 1$. In the both cases $${p-1 \over q} \ < \ a_n \ < \ {p+1 \over q}$$. 

\end{proof}

\section{Approximating intervals}
\label{intervals}

The title of the final section of RZ with which we are mainly concerned here is \emph{Infinite Quantity Concepts} (or as we have indicated this could be interpreted as \emph{Infinite Number Expressions} or similar variants).  And it lives up to this title for the first 5 sections. But thereafter (and there are more than 100 sections in the thereafter){---}with the exception of \S 21{---}there are very few references to infinite number expressions. The great majority of the working is with intervals as defined in \S 6 by two equations of the form
		$$S = \frac{p}{q} + P_1 \quad S = \frac{p+1}{q} - P_2$$
That is, with intervals of the form $[\frac{p}{q}, \frac{p+1}{q})$ for all values of natural numbers $q$. It is possible we suggest, that the pre-occupation with the sequence interpretation of infinite number expressions has, at least for some commentators, been a distraction from Bolzano's main focus.

As far as we know it has not been observed in the previous literature that there are two rather different ways of deriving infinite rational sequences from Bolzano's concept of measurable numbers. The common approach is that of \emph{partial computation} that depends on the detailed procedural evaluation of an infinite number expression: it is what we have called the \emph{sequence interpretation}. Another approach is to begin from a concept like $\sqrt{2}$, or a rational like $\frac{2}{3}$, for either of which we may derive an algorithm, or a decimal expansion, which will allow us to generate approximating intervals. Choosing the left-hand (or right-hand) endpoints of these intervals then also generates an infinite rational sequence.  In fact both these approaches are at least strongly hinted at in the original publication (Rychl\'{i}k 1962) but they are not equally taken up in the subsequent literature. Both views seem to us legitimate and significant though the sequence interpretation would have limited application to the concept of infinite number expressions in general. But it is Bolzano's analogy between an infinite number concept and a pocket watch (\S 3), and his preference for considering the latter as a single unit{---}without regard to the numerous components{---} that is highly suggestive. It supports the view that Bolzano's idea of a measurable number was that of single value which is not so much represented or calculated from a sequence but rather uniquely \emph{associated} with an infinite collection of approximating intervals.

We should therefore like to re-emphasise the approximating intervals view. It seems in fact to be the dominant view in RZ. Here we shall rely on the researches of others for some of our argument.  In the chapter (Mainzer 1990) it is reported that in the work (Bachmann 1892) there is a systematic use of nested intervals to introduce real numbers. Mainzer indicates in some detail how such an approach might be developed in modern terms. A \emph{rational net} is defined as a sequence of closed strictly nested intervals on rationals with lengths tending to zero. A net $(J_n)$ is said to be \emph{finer} than $(I_n)$ if $J_n \subseteq I_n$ for all $n$. Then two nets $(I_n)$ and $(I'_n)$ are said to be equivalent if there is a net $(J_n)$ finer than each one. He shows how then real numbers can be defined as equivalence classes of nets. It would be possible to follow this up with definitions of arithmetic operations and ordering on these classes and show they form a complete ordered field. Instead Mainzer follows a different, more interesting, strategy. He establishes a direct correspondence between on the one hand the classes of nets and Dedekind cuts, and on the other hand between the net classes and Cantor's classes of  fundamental sequences. These correspondences can be set up rather simply and reveal a satisfying underlying similarity between classes of approximating intervals and the Dedekind and Cantor constructions. Full details are given in (Mainzer 1990). We have already explained that Bolzano's approximating intervals are not, themselves, strictly nested, they are \emph{dually directed}, but this property allows us in a straightforward fashion to derive a strictly nested family of intervals which could therefore be used in such a construction as Mainzer describes. For example, one way to do this for a given measurable number $S$ is to take, instead of the approximating intervals for \emph{all} values of $q$, is to take, for a given fixed value of $q$, say $q_0$, the sub-collection of intervals for all multiples $nq_0$ for natural numbers $n > 1$. In this way each approximating interval is a subset of the previous one and they do form a nested sub-collection of all the approximating intervals. The length of these intervals tends to zero and must have the same unique common point as that of the collection of all intervals, namely that corresponding to the number $S$.

Bachmann was not the only mathematician of the late 19C to define real numbers in terms of nested intervals. The work (Burn 1992) reports that in an appendix to Volume 3 of the \textit{Cours d'analyse} (Jordan 1887) Jordan gave a construction of irrational numbers. It was also using nested intervals but in the later, more influential, editions he gave accounts similar to Dedekind cuts.

\section{Conclusions and Further Work}
\label{conclusion}

In (Rychl\'{i}k 1962) Bolzano's work is hailed in the title as a \lq theory of real numbers\rq. We have tried in this paper to focus on this claim and to give greater emphasis than in previous treatments to the key feature of measurable numbers that they are defined in terms of approximating intervals. In doing so we have inevitably neglected many interesting and relevant themes. We have not dealt properly, for example, with Bolzano's concepts of infinitely small numbers, infinitely large numbers and infinite collections. These are complicated topics, especially if they are to be considered{---}as they must{---}alongside the context of Bolzano's life and his times. 

We cannot now know any more than the evidence that remains to us allows for how exactly Bolzano regarded his infinite number concepts and the approximating intervals of his measurable numbers. So there is an essential element of surmise and speculation in any judgement we make. It has to be acknowledged that Bolzano did not explicitly refer to any such sequence interpretation as many commentators (including ourselves) have constructed. Nor did he explicitly relate measurable numbers to the collection of their approximating intervals in the way we have supportted.  But we have given{---}at least in broad outline{---}good evidence, albeit with the hindsight of a modern perspective, that Bolzano's insights into a \lq correct concept of number \rq \ did indeed constitute the core of the real number system as recognised in modern times.

Bolzano never did re-write the crucial discussion around \S 54 (as he said was needed) relating to the two views of the equality of measurable numbers. But we have highlighted this issue and drawn attention to his premonitions of the idea of equivalence relation. His careful development of many of the algebraic properties of an ordered field is also worthy of further attention.

In addition to the uncertainties surrounding work for which there was never a \lq final version\rq \ we have mentioned also some other specific factors contributing to the controversies in the literature. For example, there is the fact of the two stages (1962 and 1976) in publication of transcriptions of his manuscripts, the distractions, and difficulties, in the \lq sequence interpretation\rq \ of infinite number expressions and the association (misguided in our view) of Bolzano's infinitely small numbers with non-standard analysis.

Finally we mention a theme that calls for further investigation in a future work. The Czech dissident Petr Vop\v{e}nka working in Prague developed in (Vop\v{e}nka 1979) an Alternative Set Theory (AST). Although his motivation did not come from Bolzano's work his ideas seem to us a framework well-suited to Bolzano's mathematical work and particularly his work in relation to infinite numbers and collections.  Vop\v{e}nka's AST makes use of so-called semisets and the phenomenological notion of a \lq horizon\rq  separating finite from infinite numbers. There is an extensive theory of semisets which can be used to support a theory of numbers, and a theory of the continuum.  Some of this work appears to relate quite closely to Bolzano's ideas and results about measurable numbers

\end{document}